\documentclass[11pt]{amsart}
\usepackage{amssymb}
\usepackage{amsfonts}
\usepackage{color}

\usepackage{xy}
\xyoption{all}

\usepackage{setspace}
\newtheorem{theorem}{Theorem}[section]
\newtheorem{proposition}[theorem]{Proposition}
\newtheorem{lemma}[theorem]{Lemma}
\newtheorem{corollary}[theorem]{Corollary}

\theoremstyle{remark}
\newtheorem{remark}[theorem]{Remark}

\newtheorem{example}[theorem]{\bf Example}
\renewenvironment{proof}{{\noindent\bf Proof.}}{\hfill $\Box$\par\vskip3mm}

\newcommand{\Bb}{\mathcal{B}}

\newcommand{\Tt}{\mathcal{T}}

\def\NN{{\mathbb N}}

\def\KK{{\mathbb K}}

\begin{document}
\title[Semiartinian profinite algebras]{Semiartinian Profinite Algebras have nilpotent Jacobson radical}
%{Serial Linear Categories, Infinite Abelian Groups, Quantum Groups and Open Questions in Corepresentation Theory}
%\dedicatory{}

\begin{abstract}
%We show that a left semiartinian profinite algebra has nilpotent Jacobson radical and is right semiartinian too.
We give a method to study the finiteness of the coradical filtration of a coalgebra; as a consequence, we show that a left semiartinian profinite algebra has nilpotent Jacobson radical and is right semiartinian too. Equivalently, we show that a for a semilocal profinite algebra, T-nilpotence implies nilpotence for the Jacobson radical. This answers two open questions from \cite{INT}.
%We study the finiteness of the coradical filtration of a coalgebra; as a consequence, we show that a left semiartinian profinite algebra has nilpotent Jacobson radical and is right semiartinian too, answering two open questions from \cite{INT}.

\end{abstract}

\author{Miodrag Cristian Iovanov\\}%$^*$}
\thanks{2010 \textit{Mathematics Subject Classification}. 16T15, 18E40, 16T05}
%20K10, 16T05, 22C05, 05C25}%, 20N99, 05E10} 
%Primary 16W30; Secondary 16S90, 16Lxx, 16Nxx, 18E40}
%\thanks{$^*$}
\date{}
\keywords{profinite algebra, coalgebra, coradical filtration, torsion theory, splitting, rational module, semiartinian}
\maketitle

\section*{Introduction}% and Preliminaries}
%\noindent

In ring theory and non-commutative algebra, semiartinian rings and modules are an important class of objects, and represent a natural generalization of artinian rings and modules. Semiartinian modules are modules which can be built up from from their Loewy series. Specifically, if $M$ is a module, then $M_0=L_0(M)$ is the socle of $M$ (the sum of all simple submodules of $M$), and for each ordinal $\alpha$, $M_\alpha=L_\alpha(M)$ is defined by induction as follows: if $\alpha=\beta+1$ is a successor, then $M_\alpha/M_\beta$ is the semisimple part of $M/M_\beta$, and $M_\alpha=\bigcup\limits_{\beta<\alpha}M_\beta$ otherwise \cite{F}. $M$ is semiartinian if this Loewy series terminates at $M$, i.e. $M=M_\gamma$ for some $\gamma$. 

At the opposite pole, one finds rings whose upper (Jordan) series ``terminates" at $0$. There are many important situations in both commutative and non-commutative algebra, where $\bigcap\limits_n J^n=0$, where $J$ is the Jacobson radical of a ring $A$. For example, every commutative Noetherian ring satisfies this property by the Krull intersection theorem. The Jacobson conjecture states that for every non-commutative (two-sided) Noetherian ring, $\bigcap\limits_n J^n=0$; while this is known for some classes of Noetherian rings (with Krull dimension, commutative), it is still open in general. Another class of rings satisfying ``the Jacobson conjecture" $\bigcap\limits_n J^n=0$ is that of profinite algebras. These are inverse limits of finite dimensional algebras, and are in a way the $\KK$-algebra analogue of profinite groups; they are particular examples of pseudocompact rings introduced by P.Gabriel \cite{Gb}. Given such a profinite algebra $A$, there is always a coalgebra $C$ such that $A\cong C^*$, and so it is at the same time a pseudocompact algebra (with respect to a suitable topology). We recall that a pseudocompact algebra is an algebra with a linear topology which has a basis of two sided ideals of finite codimension, and which is complete and separated in this topology. In fact, the category of pseudocompact algebras with continuous morphisms is dual equivalent to that of coalgebras and morphisms of coalgebras; we refer to \cite{S,DNR,Gb} for details. 

In general, questions about nil and nilpotence properties of Jacobson radical of a ring are of central interest in ring theory. Recall that an ideal $I$ of a ring $R$ is called right T-nilpotent, if given any sequence $(x_n)_n$ of elements in $I$, there is $k$ such that $b_kb_{k-1}\dots b_1=0$. Several important notions in ring theory are connected to some notions of nilpotence. A ring $R$ is left semiartinian if and only if $R/J$ is semiartinian and $J$ is right T-nilpotent. A ring is right perfect (i.e. right modules have projective covers) if and only if $R/J$ is artinian (i.e. $R$ is semilocal) and $J$ is right T-nilpotent, equivalently, $R$ is semilocal and left semiartinian. Of course, T-nilpotence implies nil for every element, and nilpotence of the Jacobson radical implies nil and T-nilpotence, but deciding the converse can be a hard question in general.

Regarded as a generalization of finite dimensional algebras, and also a class of algebras where at least the (conclusion of the) Jacobson conjecture holds, pro-finite algebras is a natural place to study connections between nil and nilpotence properties for the Jacobson radical. As noticed, T-nilpotence is closely related to being semiartinian. A natural question then arises: characterize algebras which are profinite and (one-side) semiartinian. Such an algebra has both a lower Loewy series and an upper Jacobson series, and one can ask in general what can be said about such algebras. This question in particular was asked in \cite{INT}, where it was conjectured that a (left) semiartinian profinite algebra has nilpotent Jacobson radical. In other words, this conjecture states that if a profinite algebra $A$ is semilocal, then {\it T-nilpotence implies nilpotence}. The statement of this conjecture in \cite{INT} was motivated by the fact that all known examples of coalgebras $C$ for which the dual algebra $C^*$ is (left) semiartinian were those with finite coradical filtration. These algebras were studied there in connection with the following problem, called the Dickson splitting problem. Given a ring $A$, one can consider the class of left semiartinian modules, which is a subcategory of the left $A$-modules, closed under extensions, coproducts, and subquotients. This is called the Dickson torsion, and was introduced and studied in \cite{D}. The Dickson splitting problem asks to characterize when the Dickson torsion of any module (i.e. its largest semiartinian submodule) splits off. This is part of a general class of problems called splitting problems (relative to torsion theories), which have a long history of mathematical interest (see, for example, \cite{G, K1,K2,Rot,T1,T2,T3}; also, \cite{C,I1,IO,NT} for splitting problems for profinite and pseudocompact algebras). It was originally conjectured that if the Dickson torsion splits off in any left $A$-module for a ring $A$, then $A$ is left semiartinian; this was proved to be false by a counterexample in \cite{Co}. However, it was proved to hold for profinite algebras in \cite{INT}. 

A second natural question asked in \cite{INT} is whether a left semiartinian profinite algebra $A$ is also right semiartinian; a positive answer to the first question would obviously imply a positive in the second. Here we answer these questions, and our main result is:

\begin{theorem}[Theorem \ref{t.1}]
If $A$ is a profinite algebra, which is left semiartinian (i.e. semilocal and $Jac(A)$ is right T-nilpotent), then $J=Jac(A)$ is nilpotent, $A/J$ is semisimple finite dimensional, and $A$ is right semiartinian too.
\end{theorem}

In an equivalent reformulation, for a semilocal profinite algebra, T-nilpotence implies nilpotence (see also Theorem \ref{t.2}). In proving this theorem, we introduce certain tools which might prove useful in the future in the study of the behavior of the coradical filtration of a coalgebra, and specifically the finiteness of the coradical filtration. In regards to this, we note that a problem that has attracted a good amount of interest lately is a 2003 conjecture due to Andruskiewitsch and Dascalescu \cite{AD}, which states that a co-Frobenius Hopf algebra $H$ (equivalently, a Hopf algebra with non-zero integral) has finite coradical filtration. This was known for pointed Hopf algebras from \cite{RH}, and it was proved in full generality and also for certain tensor categories (of subexponential growth) very recently in \cite{ACE} (see also \cite{AC} for a related finiteness statement; in fact, a Hopf algebra has a non-zero integral if and only if it has finite coradical filtration by the results of \cite{AD,ACE}). The problem is left open in \cite{ACE} for general Frobenius tensor categories, that is, for co-Frobenius coalgebras for which the category of comodules has a tensor structure. Hence, methods for deciding the finiteness of the coradical filtration could be of use in other situations. In particular, the methods proposed here (and used for the above mentioned main result), are contained in Proposition \ref{p.2.1} which concerns linear algebra properties of the comultiplication and tensor products, and Lemma \ref{l.unbounded} regarding a general property of sequences $(x_n)_n$ of elements in $C$, such that $(x_n)_n\not\subset C_k$ for any $k$. The proofs involve linear algebra arguments, and the exposition is kept to a general level, and only the basic results on coalgebras and their comodules are used.%, which amount mostly to linear algebra considerations.

\section{Semiartinian Profinite Algebras}

Let $A$ be a profinite algebra, that is $A=\lim\limits_{\stackrel{\leftarrow}{i}}A_i$ is an inverse limit of finite dimensional algebras $A_i$. Then there is a coalgebra $C$ such that $A=C^*$. It is proved in \cite{INT} that if the Dickson torsion part of any left $C^*$-module $M$ splits off in $M$, then $C^*$ is itself a torsion object with respect to this theory. That is, $C^*$ is left semiartinian. We aim to prove that furthermore, if $C^*$ is left semiartinian, then $C$ has finite coradical filtration, answering thus the questions left open in \cite{INT}.

For technical details on coalgebras and comodules, the and coradical filtration of a coalgebra, we refer to \cite{DNR}, and for basic module theory, to \cite{F}. Let $(C,\Delta,\varepsilon)$ be a coalgebra, $C^*$ its dual algebra, and let $M$ be a right $C$-comodule. Then $M$ is naturally a left $C^*$-module, which is a rational $C^*$-module. Recall that $M$ has a filtration $M_0\subseteq M_1\subset \dots \subset M_n\subset \dots$ called the coradical filtration, which is the Loewy series of $M$ regarded as a left $C^*$-module. Specifically, let $M_0$ be the socle of $M$, i.e the largest semisimple $C^*$-submodule of $M$; this is the first term $L_1(M)$ of the Loewy filtration of $M$. Inductively, if $L_n(M)=M_{n-1}$ is constructed then let $L_{n+1}(M)=M_{n}$ be the sub(co)module of $M$ such that $M_{n}/M_{n-1}$ is the socle of $M/M_{n-1}$. Thus, $M_n=L_{n+1}(M)$, the $n+1$'st term of the Loewy filtration of $M$ regarded as a left $C^*$-module. It is well known that for a right $C$-comodule $M$ (equivalently, a left rational $C^*$-module $M$), one has $\bigcup\limits_{n}M_n=M$. We denote by $C_0\subset C_1\subset C_2 \dots $ the coradical filtration of $C$; this is the same when $C$ is considered as a left $C$-comodule or as a right $C$-comodule, and the spaces $C_n$ are consequently subcoalgebras of $C$ (see \cite{DNR}). 

For a right $C$-comodule $M$, we write $lw(M)$ for the Loewy length of $M$, that is, $lw(M)=n+1$ if $n$ is the smallest number for which $M=M_n$; otherwise, we write $lw(M)=\infty$. Let $x\in C$. We write $lw(x)=lw(C^*\cdot x\cdot C^*)$ for the Loewy length (i.e. length of the coradical filtration) of the (finite dimensional) coalgebra generated by $x$. It is easy to see that in fact $lw(x)=lw(C^*\cdot x)=lw(x\cdot C^*)$, which justifies the notation. Note also that $lw(x)=n+1$ if and only if $x\in C_n\setminus C_{n-1}$ (since $C_n=L_{n+1}({}_{C^*}C)=L_{n+1}(C_{C^*})$).

For an element $x\in C_n$, we have $\Delta(x)=x_1\otimes x_2\in C_0\otimes C_n+C_1\otimes C_{n-1}+\dots C_{n}\otimes C_0$.  We next observe the following: for each $i+j=n$, $i,j\geq 0$, in any fixed tensor representation of $\Delta(x)=x_1\otimes x_2$ there is always a tensor monomial which is ``essentially'' in $C_i\otimes C_j$; this is somewhat folklore, but we include a short proof for sake of  completeness.

\begin{proposition}\label{p.2.1}
Let $x\in C$ such that $lw(x)>n$ (equivalently, $x\notin C_{n-1}$) and let $\Delta(x)=x_1\otimes x_2=\sum\limits_{i=1}^ta_i\otimes b_i$. Then for each $1\leq k\leq n$, there is some $i$ such that $lw(a_i)\geq k$ and $lw(b_i)\geq n-k$.
\end{proposition}
\begin{proof}
Assume otherwise; this means that for each $i$, either $lw(a_i)< k$ so $a_i\in C_{k-1}$ or $lw(b_i)< n-k$ so $b_i\in C_{n-k-1}$. In either case, we get $a_i\otimes b_i\in C_{k-1}\otimes C+C\otimes C_{n-k-1}=C_{k-1}\wedge C_{n-k-1}$. Hence $\Delta(x)\in C_{k-1}\wedge C_{n-k-1}=C_0^{\wedge k}\wedge C_0^{\wedge n-k}=C_0^{\wedge n}=C_{n-1}$ (see e.g. \cite[Section 3.1]{DNR}), which is a contradiction to $x\notin C_{n-1}$. 
%$x_1\otimes x_2\in C_0\otimes C_n+\dots C_{k-1}\otimes C_{n-k+1}+(C_{k-1}\otimes C_{n-k}+C_k\otimes C_{n-k-1})+C_{k+1}\otimes C_{n-k-1}+\dots C_n\otimes C_0$, because the terms in $C_k\otimes C_{n-k}$ must be in $C_{k-1}\otimes C_{n-k}+C_k\otimes C_{n-k-1}$ in this situation. Then, we see that $xC_{k-1}^\perp\in C_{n-k-1}$. Hence $lw(xC_{k-1}^\perp)=\leq n-k$. Also note that $(xC^*/xC_{k-1}^\perp)\cdot C_{k-1}^\perp=(xC^*/xC^*C_{k-1}^\perp)\cdot C_{k-1}^\perp=0$, which shows that $xC^*/xC_{k-1}^\perp$ is a left $C$-comodule which has an induced structure as a $C_{k-1}$-comodule and so we have $lw(xC^*/xC_{k-1}^\perp)\leq k$. Therefore, $lw(xC^*)\leq lw(xC^*/xC_{k-1}^\perp)+lw(xC_{k-1}^\perp)\leq k+n-k=n$ and so $x\in C_{n-1}$, a contradiction. 
\end{proof}

\begin{remark}
For $x\in C$ let $(a_k)_{k=1,\dots,n}$ be a (finite) basis for $C^*\cdot x$ and write $\Delta(x)=\sum\limits_{k}a_k\otimes b_k$. Then $(b_k)_k$ are linearly independent. Indeed, if there is $b_t=\sum\limits_{k\neq t}\lambda_kb_k$ then $\Delta(x)=\sum\limits_{k\neq t}a_k\otimes b_k+\sum\limits_{k\neq t}a_t\otimes \lambda_kb_k=\sum\limits_{k\neq t}(a_k+\lambda_ka_t)\otimes b_k$. This shows that for each $c^*\in C^*$ we have $c^*\cdot x=\sum\limits_{k\neq t}c^*(b_k)(a_k+\lambda_ka_t)$, and so $C^*\cdot x\subseteq {\rm Span}\{a_k+\lambda_ka_t|k\neq t\}$ and this shows that $\dim(C^*\cdot x)<n$, a contradiction.
\end{remark}

For $X\subseteq C$, we use the notation $X^\perp \subseteq C^*$ for the set $X^\perp=\{c^*\in C^* | c^*(X)=0\}$. Let $J=C_0^\perp$; then $J$ is the Jacobson radical of $C^*$.

\begin{lemma}\label{l.unbounded}
Let $(x^n)_n$ be a family of elements in $C$, and such that $(lw(x^n))_n$ is unbounded. Then there is $a^*\in C^*$ such that $a^*\in C_0^\perp$ and $(lw(a^*\cdot x^n))_n$ is unbounded.
\end{lemma}
\begin{proof}
We may obviously find a subsequence $y^n=x^{k_n}$ of $(x^n)$ such that $lw(y^{n+1})>2lw(y^n)+2$, and $lw(y^1)>2$. In particular, $lw(y^n)> 2^{n}$. Also pick $y^0\in C_0$, $y^0\neq 0$, so $lw(y^1)>2lw(y^0)+2$. We construct a set $\Bb$ as follows. First, let $\Bb_0$ be a basis of $C_0$. For each $n\geq 1$, we pick a basis $a^n_{ji}$ of $C^*\cdot y^n$ such that for each $j$, $(a^n_{ji})_i$ is a basis in $L_j(C^*\cdot y^n)/L_{j-1}(C^*\cdot y^n)=C^*\cdot y^n\cap C_j/C^*\cdot y^n\cap C_{j-1}$. Thus, $lw(a^n_{ji})=j+1$. Let $\Delta(y^n)=\sum\limits_{j,i}a^n_{ji}\otimes b^n_{ji}$. Then the $(b^n_{ji})_{i,j}$ are linearly independent too by the previous remark. For $c^*\in C^*$, we have $c^*\cdot y^n=\sum\limits_{j,i}c^*(b^n_{ji})a^n_{ji}$. We note that if $c^*(b^n_{ks})\neq 0$ for some $k$ and some $s$ then $c^*\cdot y^n\notin C_{k-1}$. Indeed, the choice of the $a^n_{ji}$'s shows that the set $\{a^n_{ji}|j\leq k-1\}$ is a basis of $C^*\cdot y^n\cap C_{k-1}$, and so if $c^*\cdot y^n\in C_{k-1}$ it must be a linear combination of the elements $\{a^n_{ji}|j\leq k-1\}$, and so $c^*(b^n_{ks})=0$ in this case. By Proposition \ref{p.2.1}, since $lw(y^n)>2lw(y^{n-1})+2$, there is some element $z^n=b^n_{ks}$ for some $k$ and some $s$ such that $lw(a^n_{ks})\geq lw(y^{n-1})+1$ and $lw(z^n)=lw(b^n_{ks})\geq 2lw(y^{n-1})+2 - lw(y^{n-1})-1= lw(y^{n-1})+1$ (of course, $k,s$ depend on $n$). Because for each $n$ the $a^n_{ji}$'s are independent, it follows that $b^n_{ks}\in y^n\cdot C^*$ (to see this, for example, pick $c^*\in C^*$ equal to $1$ on $a^n_{ks}$ and $0$ on the other $a^n_{ji}$'s and then $b^n_{ks}=y^n\cdot c^*$); therefore $lw(z^n)\leq lw(y^n)$. \\
Now let $\Bb=\Bb_0\cup\{z^n|n\geq 1\}$. Note that $z^1\notin {\rm Span}(\Bb_0)=C_0$ since $lw(z^1)\geq lw(y^0)+1>0$. By the above equalities, we see that $lw(z^{n-1})\leq lw(y^{n-1})<lw(z^n)$ so $lw(z^n)>lw(z^{n-1})$ for $n>1$. Hence, it follows that $z^n\notin {\rm Span}(\Bb_0\cup\{z^1,z^2,\dots,z^{n-1}\})$, since otherwise as $z_i\in C_{lw(z^{i})-1}\subseteq C_{lw(z^{n-1})-1}$ for $i<n$ it would follow that $z^{n}\in C_{lw(z^{n-1})-1}$, i.e. $lw(z^n)\leq lw(z^{n-1})$. Therefore, we get that the elements of $\Bb$ are linearly independent. Hence, we can find $a^*\in C^*$ such that $a^*\vert_{B_0}=0$ and $a^*(z^n)=1$ for all $n$. Then $a^*\in C_0^\perp$, and for each $n$, since $z^n=b^n_{ks}$ and $a^*(z^n)=a^*(b^n_{ks})\neq 0$, by the above remarks we have $a^*\cdot y^n\notin C_{k-1}$, i.e. $lw(a^*\cdot y^n)\geq k$. But since by construction $k+1=lw(a^n_{ks})\geq lw(y^{n-1})+1>2^{n-1}+1$, we get $lw(a^*\cdot y^n)\geq k>2^{n-1}$, and so $(lw(a^*\cdot y^n))_n$ is unbounded. In particular, since $(y^n)_n$ is a subsequence of $(x^n)_n$, the proof is finished. 
\end{proof}

We note that this is what is needed to answer the conjecture of \cite{INT}. %However, we see that we can give a much shorter aproach to the result of \cite{INT} and prove it alltogether. 
For any left $C^*$-module $M$ let us denote $\Tt_D(M)$ the Dickson torsion (i.e. the semiartinian part) of $M$. We recall from \cite{INT} that if $\Tt_D(M)$ splits off in $M$ for any left $C^*$-module $M$ then $C_0$ is finite dimensional, and $C^*$ is necessarily a left semiartinian ring. %and then a  $C^*$-module is semisimple if and only if it is canceled by $J$ and in this case, since $C_0$ is finite dimensional, it is also a rational module. Moreover, for a module $M$ of finite Loewy length $M$, when $C_0$ is finite dimensional, $lw(M)=\min\{n|J^n\cdot M\}$ (this statement is always true for rational modules; \cite[Lemma 2.2]{I2}).

\begin{theorem}\label{t.1}
%Let $C$ be a coalgebra such that the Dickson torsion $\Tt_D(M)$ (i.e. the semiartiniain part) of any left $C^*$-module $M$ splits off. Then $C$ has finite coradical filtration.
Let $C$ be a coalgebra such that $C^*$ is left semiartinian. Then $C$ has finite coradical filtration.
\end{theorem}
\begin{proof}
By contradiction, assume otherwise. Then there is $(x^n)_n\subset C$ a sequence of elements of $C$ such that $lw(x^n)=n$ (i.e. $x^n\in C_n\setminus C_{n-1}$, $n\geq 1$). Let $M=\prod\limits_nC^*\cdot x^n$.  We prove  that $\Tt_D(M)=\{(y^n)_n\in M|(lw(y^n))_n{\rm\,is\,bounded}\}$. We have $(lw(y^n))_n$ is bounded if and only if there is $k$ such that $J^k\cdot y^n=0,\,\forall n$, equivalently, $J^k\cdot (y^n)_n=0$ ($J=C_0^\perp$). This is equivalent to saying that $C^*\cdot (y^n)_n$ has finite Loewy length, i.e. $(y^n)_n\in L_\omega(M)$ ($\omega$ is the first countable ordinal). Thus $L_\omega(M)=\{(y^n)_n|(lw(y^n))_n{\rm\,is\,bounded}\}$. \\
Now, to prove $\Tt_D(M)=L_\omega(M)$, assume, by contradiction, that there is $y=(y^n)_n\in \Tt_D(M)\setminus L_\omega(M)$. Then the Loewy length $lw(y)$ of $y$ (i.e. of $C^*\cdot y$) is at least $\omega+1$ . Since $C^*y/L_\omega(C^*y)\neq 0$, there is a simple submodule of $C^*y/L_\omega(C^*y)$, and let this submodule be generated by $c^*\cdot y$ modulo $L_\omega(M)$. Thus $c^*\cdot y\notin L_\omega(M)$ so $(lw(c^*\cdot y^n))_n$ is unbounded, but $J\cdot c^*\cdot y\subseteq L_\omega(M)$. Hence, for any $a^*\in J$, $(lw(a^*\cdot c^*y^n))_n$ is bounded. But this is a situation which is in contradiction to Lemma \ref{l.unbounded}.\\
Finally, note that $M$ is not semiartinian, since $(lw(x^n))_n$ is unbounded, so $(x^n)_n\notin L_\omega(M)=\Tt_D(M)$. This is a contradiction to $C^*$ left semiartinian (since if $C^*$ is left semiartinian, then any left $C^*$-module is semiartinian). 
\end{proof}

We first recall from \cite[Proposition 1.1]{INT} that if $A$ is a profinite algebra, so $A=C^*$ for some coalgebra $C$, and  $A$ is left semiartinian, then $A$ has only finitely many isomorphism types of simple modules. Indeed, let  $(S_i)_i$ is a set of representatives for the isomorphism types of the simple left $C$-comodules, and let $P=\prod\limits_{i\in I}S_i^*$ as a left $A$-module. Then it is not difficult to see that $\Tt_D(P)=\bigoplus\limits_{i\in I}S_i^*$, so $\Tt(P)=P$ only if $I$ is finite. (in fact, $\Tt(P)$ is a direct summand in $P$ only if $I$ is finite). Combining the above theorem with the results of \cite{INT} we have

\begin{corollary}\label{c.fin}
$A$ be a profinite algebra, so $A=C^*$ where $C$ is a coalgebra. The following assertions are equivalent:\\
(i) The semiartinian submodule (Dickson torsion) of every left $A$-module splits off.\\
(ii) The semiartinian submodule (Dickson torsion) of every right $A$-module splits off.\\
(iii) $A$ is left semiartinian.\\
(iv) $A$ is right semiartinian.\\
(v) $C$ is almost finite (i.e. $C_0$ is finite dimensional, equivalently, there are only finitely many isomorphism types of simple $A$-modules) and $C$ has finite coradical filtration, equivalently, the Jacobson  radical of $A=C^*$ is nilpotent. 
\end{corollary}
\begin{proof}
(i) $\Rightarrow$ (iii) By \cite[Theorem 3.3]{INT}, if the semiartinian submodule of every  left $A$-module splits off, then $A$ is left semiartinian. \\
(iii) $\Rightarrow$ (v) As noted above, by \cite[Proposition 1.1]{INT} $C$ has only finitely many types of isomorphism of simple comodules, so $C_0$ is finite dimensional. By the previous Theorem \ref{t.1}, $C$ has finite coradical filtration. \\
(v) $\Rightarrow$ (i) is obvious, since if $C_0$ is finite dimensional, then $A/J$ is semisimple, and $J^k/J^{k+1}$ is semisimple for all $k$ (since it is annihilated by $J$). Since $J^n=0$, $0\subseteq J^{n-1}\subseteq \dots J\subseteq A$ is a finite filtration of $A$ with semisimple modules, so $A$ is semiartinian.\\
(i) $\Leftrightarrow$ (ii) $\Leftrightarrow$ (iv) follow by symmetry.
\end{proof}

In regard to the algebras characterized by the above Corollary, it is natural to ask what other ring theoretical properties does it have. More generally, one can ask about ring theoretical properties involving nil or nilpotence of the Jacobson radical for a profinite algebra. For example, if $C$ is a coalgebra, then idempotents always lift modulo the Jacobson radical of $C^*$ (see \cite[Proposition 2.2.2]{R} and \cite[Remark 1.8]{DIT}). This raises the question of when is $C^*$ semiperfect. Using this, it is shown in \cite{DIT} that $C^*$ is semiperfect if and only if $C$ is almost finite, equivalently, $C^*$ is semilocal. We note that this also follows directly: : if $C^*$ is semilocal, then every simple module is rational, since $C^*/J$ is finite dimensional semisimple. By \cite[Lemma 1.4]{I}, if $A=C^*$ is a profinite algebra dual to the coalgebra $C$, every simple rational left $C^*$-module $S$ has a projective cover; it is obtained via the canonical morphism $E(S^*)^*\rightarrow S$. This shows that $C^*$ is semiperfect by an equivalent.

A basic question would be to understand simple modules over such an algebra $A=C^*$. A class of simple modules of particular interest are, of course, the rational simple modules. Quite interestingly also vis-a-vis of the semiperfect property, by \cite[Proposition 1.9]{DIT}, the simple rational modules over $C^*$ are distinguished by the following property: they are the only simple $C^*$-modules which have a projective cover. At the same time, if $C^*$ is not semilocal, equivalently, there are infinitely many types of isomorphism of simple rational modules, then by \cite[Lemma 6.1]{HIT} it must also have simple non-rational modules. We can thus summarize all these results in \cite{DIT,HIT,INT} and in the above Theorem \ref{t.1} in the following 

\begin{theorem}\label{t.2}
(I) Let $A$ be a profinite algebra, and let $C$ be a coalgebra such that $A=C^*$. The following are equivalent:\\
(i) $A$ is semilocal.\\
(ii) Every simple $A$-module is rational.\\
(iii) $A$ is semiperfect. \\
(II) Let $A$ be a profinite algebra. Then the following statements are equivalent:\\
(i) $A$ is left or right semiartinian.\\
(ii) $A$ is left or right perfect.\\
(iii) $A$ is semilocal and semiprimary, i.e. the Jacobson radical of $A$ is nilpotent.\\
(iv) $A$ is semilocal and $Jac(A)$ is left or right T-nilpotent.
\end{theorem}
\begin{proof}
(I) follows from the comments preceding the theorem. For (II), the equivalence (i)$\Leftrightarrow$(iii) is given by Corollary \ref{c.fin}. For (i)$\Leftrightarrow$(ii), note that a ring is left perfect if and only if it is left semiartinian and semilocal. Since a left semiartinian profinite algebra is semilocal, the equivalence follows.\\
(i) $\Leftrightarrow$ (iv) follows since a ring $R$ is left semiartinian if and only if $R/Jac(R)$ is semiartinian and $Jac(R)$ is right T-nilpotent. Hence, if $A$ is semiartinian, by (iii) $A$ is semilocal and $J=Jac(A)$ is nilpotent and then also T-nilpotent; conversely, if $A$ is semilocal then $A/J$ is semiartinian, and since $J$ is left (or right) T-nilpotent, $A$ is left (or right) semiartinian. 
\end{proof}

We end by considering a few examples. We first recall the following useful construction. Consider a quiver $Q$, which consists of a set of arrows $Q_0$ and a set of arrows $Q_1$, together with two functions $s,t:Q_1\rightarrow Q_0$, called ``source'' and ``target''. Note that the set of vertices and arrows are not necessarily finite, and loops and infinitely many arrows between two vertices are allowed. A path in $Q$ is a sequence of arrows $x_1x_2\dots x_n$ such that $t(x_{k})=s(x_{k+1})$ (we note that sometimes the ``composition'' convention is used, i.e. $t(x_{k+1})=s(x_k)$). We denote $|p|=n$ the length of $p$. The path algebra (or quiver algebra) over a field $\KK$ of this quiver is the $\KK$ vector space with a basis consisting of all paths in $Q$ and multiplication given by concatenation of paths: if $p=x_1\dots x_n$ and $q=y_1\dots y_t$ then $p*q=pq=x_1\dots x_ny_1\dots y_t$ if $t(x_n)=s(y_1)$ (i.e. $p$ ends where $q$ starts) and $p*q=pq=0$ otherwise. We denote this algebra by $\KK[Q]$. There is also a coalgebra structure on this space with a basis of paths, called the path coalgebra (or quiver coalgebra) with comultiplication and counit given by
\begin{eqnarray*}
\Delta(p) & = & \sum\limits_{p=qr} q\otimes r\\
\varepsilon(p) & = & \delta_{|p|,0}
\end{eqnarray*}
We denote this coalgebra by $\KK Q$. More generally, let $S$ be a set of paths in $Q$, which is closed under taking subpaths. Then the span $C$ of $S$ inside $\KK Q$ is easily seen to be a subcoalgebra of $\KK Q$. Such a coalgebra is called a monomial coalgebra (or path subcoalgebra). The $n$'th term $C_n$ of the coradical filtration of this coalgebra $C$ has a basis consisting of the paths in $S$ of length at most $n$. The simple left and right $C$-comodules are spanned by vertices (i.e. paths of length $0$), and the injective hulls of such a simple right comodule $\KK a$ is spanned by all paths in $S$ that start at $a$ (see \cite{simson}, and the proof of \cite[Proposition 2.5]{simson}). We can now introduce our examples.

\begin{example}
Let $Q$ the ``thick arrow quiver'', i.e. the quiver having two vertices $a,b$ and infinitely many arrows $x_n$ between $a$ and $b$:
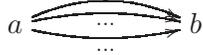
$$\xymatrix{
a \ar@/^2ex/[rr] \ar@/^1ex/[rr]_\dots \ar@/^-1ex/[rr]_{\dots} & & b
}$$ 
Let $C$ be its path coalgebra. It has a basis consisting of $\{a,b\}\cup \{x_n|n\geq 1\}$ and comultiplication $a\rightarrow a\otimes a$, $b\rightarrow b\otimes b$, $x_n\rightarrow a\otimes x_n+x_n\otimes b$ and counit $\varepsilon(a)=\varepsilon(b)=1$, $\varepsilon(x_n)=0$. Since all paths in $Q$ have length $\leq 1$, we have $C=C_1$, so $J^2=0$. Moreover, $C$ is semilocal, since $C_0$ is spanned by $a$ and $b$. Hence, it satisfies the equivalent conditions of Corollary \ref{c.fin}, and of Theorem \ref{t.2}. The algebra $C^*$ is profinite, semilocal, semiartinian and perfect on both sides. It is not finite dimensional.\\
It a standard exercise and perhaps interesting to note that the algebra $C^*$ can be thought of as an upper triangular matrix algebra as follows:
$$C^*\cong \left(\begin{array}{cc}
	\KK & \KK^{\NN}\\
	0 & \KK
\end{array}\right)$$
(here $\KK^\NN=(\bigoplus\limits_\NN\KK)^*$).
\end{example}

\begin{example}
Let $Q$ be a quiver with finitely many vertices, and $S$ a set of paths in $Q$ which is closed under subpaths. Assume that the lengths of paths in $S$ is bounded by some number $N$. Let $C$ be the monomial coalgebra with basis $S$. Then $C=C_N$, so the coradical filtration of $C$ is finite. Moreover, $C_0$ is finite, so again $A=C^*$ is profinite, semilocal, semiartinian and perfect on both sides.
\end{example}

\begin{example}
Let $Q$ be a quiver with no oriented cycles, and finitely many vertices (note: infinitely many arrows between two vertices are allowed). Let $N=|Q_0|<\infty$. Then for any path $p$ in $Q$, we have $|p|\leq N$. Otherwise, if $|p|=N+1$ for some path then by the pigeonhole principle there is a repetition of vertices in $p$. The subpath of $p$ starting and ending at such a repeated vertex would then be an oriented cycle. Therefore, the path coalgebra $C$ of this quiver is almost finite (so $C_0$ is finite dimensional) and has finite coradical filtration. As before, this coalgebra satisfies the conditions of Corollary \ref{c.fin}, and of Theorem \ref{t.2}.\\
We note that to come up with such a quiver it is enough to start with some finite quiver with no oriented cycles, then label each arrow $\alpha:a\rightarrow b$ with a set $S_\alpha$, and construct a quiver $Q'$ obtained from the quiver $Q$ by replacing each arrow $\alpha:a\rightarrow b$ by a set of cardinality $|S_\alpha|$ (possibly infinite) of arrows $a\rightarrow b$. 
\end{example}

%Recall that a ring $R$ is left semiartinian if and only if $R$ is right T-nilpotent and $R/Jac(R)$ is semiartinian. The above Corollary states that for profinite algebras, the T-nilpotence of the Jacobson radical is very close to its nil

\begin{center}
\sc Acknowledgment
\end{center}
This work was supported in part by the strategic grant POSDRU/89/1.5/S/58852, Project ``Postdoctoral programe for training scientific researchers'' cofinanced by the European Social Fund within the Sectorial Operational Program Human Resources Development 2007-2013.\\
Tha author acknowledges the careful remarks of the referee, which greatly improved the presentation of this paper.
%The author wishes to dedicate this work to the memory of his former advisor Samuel D. Shack

%\bigskip\bigskip\bigskip

%\begin{center}
%\sc Acknowledgment
%\end{center}

%\newpage

\vspace*{3mm} 
\begin{flushright}
\begin{minipage}{148mm}\sc\footnotesize

Miodrag Cristian Iovanov\\
University of Iowa\\
Department of Mathematics, McLean Hall \\
Iowa City, IA, USA\\
%Miodrag Cristian Iovanov\\
%University of Southern California\\
%Department of Mathematics, 3620 South Vermont Ave. KAP 400C \\
%Los Angeles, California 90089-2532\\
and\\
University of Bucharest, Faculty of Mathematics, Str. Academiei 14\\ 
RO-010014, Bucharest, Romania\\
%State University of New York (Buffalo)\\
%Department of Mathematics, 244 Mathematics Building\\
%Buffalo, NY 14260-2900, USA\\
{\it E--mail address}: {\tt
yovanov@gmail.com; iovanov@usc.edu}\vspace*{3mm}

\end{minipage}
\end{flushright}
\end{document}